\documentclass[11pt,oneside]{amsart}
\usepackage{amsmath,ifthen, amsfonts, amssymb, srcltx,
   amsopn, textcomp}
\usepackage[all]{xy}

\usepackage{hyperref}
\usepackage{graphicx}

\usepackage{dutchcal}

\usepackage[small]{caption}

\newcommand{\showcomments}{yes}

\newcommand{\hidetodo}[1]
{\ifthenelse{\equal{\showcomments}{yes}}%
{#1}
}

\newsavebox{\commentbox}
{\ifthenelse{\equal{\showcomments}{yes}}%
{\footnotemark
        \begin{lrbox}{\commentbox}
        \begin{minipage}[t]{1.25in}\raggedright\sffamily\tiny
        \footnotemark[\arabic{footnote}]}
{\begin{lrbox}{\commentbox}}}%
{\ifthenelse{\equal{\showcomments}{yes}}%
{\end{minipage}\end{lrbox}\marginpar{\usebox{\commentbox}}}
{\end{lrbox}}}

\newtheorem{thm}{Theorem}[section]
\newtheorem{lem}[thm]{Lemma}

\newtheorem{cor}[thm]{Corollary}

\theoremstyle{definition}
\newtheorem{defn}[thm]{Definition}
\newtheorem{rem}[thm]{Remark}
\newtheorem{exmp}[thm]{Example}

\newtheorem{prob}[thm]{Problem}

\newcommand{\field}[1]{\mathbb{#1}}

\newcommand{\reals}{\ensuremath{\field{R}}}
\newcommand{\peals}{\ensuremath{\field{P}}}

\newcommand{\boundary}   {{\ensuremath \partial}}

\newcommand{\p}{\textup{\textsf{p}}}

\newcommand{\edges}{\mathcal{E}}

\begin{document}

\title[Good Stackings, Bislim Structures, and Invariant Staggerings]{Good Stackings, Bislim Structures,\\ and Invariant Staggerings}

\author[J.~Bamberger]{Jacob Bamberger}
\author[D.~Carrier]{David Carrier}
\author[J.~Gaster]{Jonah Gaster}
\author[D.~T.~Wise]{Daniel T. Wise}

\address{Dept. of Math. \& Stats.\\
McGill Univ. \\
Montreal, QC, Canada H3A 0B9 }

\email{jacob.bamberger@mail.mcgill.ca}
\address{Dept. of Math. \& Stats.\\
McGill Univ. \\
Montreal, QC, Canada H3A 0B9 }

\email{david.carrier@mail.mcgill.ca}
\address{Dept. of Math. Sci. \\
University of Wisconsin-Milwaukee\\
Milwaukee, WI}

\email{gaster@uwm.edu}
\address{Dept. of Math. \& Stats.\\
McGill Univ. \\
Montreal, QC, Canada H3A 0B9 }

\email{wise@math.mcgill.ca}
\date{April 6, 2020}
\thanks{Research supported by NSERC}

\maketitle

\begin{abstract}
Two seemingly different properties of 2-complexes were developed concurrently as criteria for the nonpositive immersion property: `good stackings' and `bislim structures'.
We establish an equivalence between these properties by introducing a third property `invariant staggerings' mediating between them.
\end{abstract}

\section{Introduction}
The goal of this paper is to establish the equivalence between \emph{good stackings} ($GS$) and \emph{bislim structures} ($BS$).
These properties were introduced by Louder-Wilton and Helfer-Wise to solve a cycle-counting problem associated to one-relator groups \cite{LouderWilton2014,HelferWise2015}.
To facilitate this equivalence, we were led to the notion of an \emph{invariant staggering} ($IS$).
Staggered complexes arise naturally 
in the theory of one-relator groups; 
see Howie's beautiful exposition \cite{Howie87}.
An invariant staggering is an extension of the classical notion but differs in two respects: first, invariant staggerings live in the setting of the universal cover and are $\pi_1$-invariant, and second, whereas staggered complexes involve a total order on cells in the quotient, invariant staggerings utilize a partial-order.

\begin{thm}\label{main thm}
Let $X$ be a 2-complex with countably many 2-cells. 

The following are equivalent:
\begin{enumerate}
\item $X$ has a good stacking.
\item $X$ has a bislim structure.
\item $X$ has an invariant staggering.
\end{enumerate}
\end{thm}

While $\pi_1X$ may not have torsion in the presence of any of the above equivalent conditions, there is a version of this theorem that holds in the presence of torsion (see Remark~\ref{rem:torsion}).
As for the assumption about 2-cells, while good stackings cannot exist when there are too many 2-cells at a vertex, there are more 
general versions, in which $\reals$ is replaced by any totally ordered set, that can be used in the same way. 

Since $\pi_1X$ is actually orderable when $X$ admits any of the above structures, we were led to consider variants that involve a total ordering.
The notion of $IS$ arose in a totally ordered form, $TIS$, and we found that $TIS$ is equivalent to a totally ordered form of $BS$, called $TBS$.
The $TIS$ version remains a motivating simplified version of $IS$. We are unsure if they are equivalent.

The proof of Theorem~\ref{main thm} consists of three logical implications contained in Lemmas~\ref{lem:GS is BS}, \ref{lem:BS implies IS}, and \ref{lem:ES to GS}.
Of particular interest are Lemma~\ref{lem:acyclic disks}, in which a good stacking of a disk diagram yields an acyclic directed dual graph, and Lemma~\ref{lem:ES to GS}, which builds a good stacking by turning the universal cover into $\pi_1$-equivariant `stepping stones' which are then `levelled' into a $\pi_1$-invariant good stacking.

\section{Good Stacking}
Let $X$ be a 2-complex with 2-cells $\{r_i\}$ having boundary paths
$\boundary_\p r_i = w_i$, and
where each attaching map  $\phi_i: w_i\rightarrow X^1$ is a closed immersed cycle.
Let $W=\sqcup w_i$ and let  $\Phi:W\rightarrow X^1$ be induced by $\{\phi_i\}$.
Let $\rho_X$ and $\rho_\reals$ be the projections of $X^1\times \reals$ onto $X^1$
and $\reals$.

\begin{defn}[$GS$]
A \emph{stacking} of $X$ is an injection $\Psi:W\rightarrow X^1\times \reals$
so that  $\Phi=\rho_X\circ \Psi$.
 The stacking is \emph{good} if for each $i$, there exists $h_i,\ell_i \in w_i$
so that:
 $\rho_\reals(\Psi(\ell_i))$
 is lowest in 
 $\rho_\reals(\Psi(\Phi^{-1}(\Phi(\ell_i))))$
 and 
  $\rho_\reals(\Psi(h_i))$ is highest in $\rho_\reals(\Psi(\Phi^{-1}(\Phi(h_i))))$.
 \end{defn}

The reader can think of the stacking as embedding $\sqcup w_i$ within $X^1\times \reals$ so that each $w_i $ is positioned so that it projects to $w_i\rightarrow X^1$, and goodness means that each $w_i$ is `visible' from both the top and bottom, without being blocked by some $w_j$.
See Figure~\ref{fig:GoodStacking} for a simple example where the 2-complex $X$ has only two 2-cells.

Good stackings were introduced and employed in a beautiful way by
$\ell$.Louder and $\mathcal h$.Wilton to prove the nonpositive immersion property for 2-complexes
associated to one-relator groups
\cite{LouderWilton2014}.

\begin{figure}
  \centering
  \includegraphics[width=.3\textwidth]{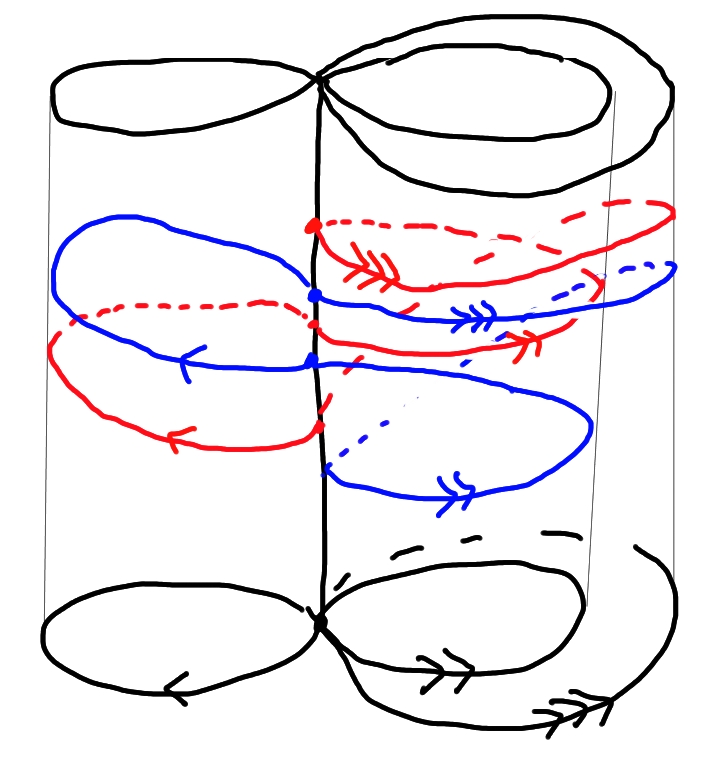}
  \caption{\label{fig:GoodStacking} The figure above illustrates a good stacking of the 2-complex associated to the presentation $\langle a,b,c \mid acb, abc\rangle$.}
\end{figure}

\section{Bislim}

\subsection{Bislim $2$-complexes}
A \emph{preorder}  $\preceq$ on a set $\edges$ is a reflexive, transitive relation on
$\edges$. 
We employ $a\prec b$ to mean 
$(a\preceq b)\wedge\neg(b\preceq a)$.

\begin{defn}[$BS$]\label{defn:slim}
  A combinatorial $2$-complex $X$, whose 2-cells have immersed attaching maps, is \emph{bislim} if:
  \begin{enumerate}
  \item\label{slim:order} There is a $\pi_1X$-invariant preorder $\preceq$ on
    the 1-cells of $\widetilde X$.
  \item\label{slim:trav once} $\boundary r$ has 
    distinguished 1-cells $r^+$ and $r^-$ for each $2$-cell $r$ of
    $\widetilde X$.\\ Moreover, $r^+$ is traversed exactly once by $\boundary_\p r$.
    \item $g(r^\pm)=(gr)^\pm$ for each $g\in\pi_1X$ and each 2-cell $r$.
  \item\label{slim:distinct} If $r_1$ and $r_2$ are distinct $2$-cells 
    in $\widetilde X$ and $r_1^+ \subset \boundary r_2$ then
    $r_1^+ \prec r_2^+$.
\item\label{bislim:bislim}
If $r_1$ and $r_2$ are distinct $2$-cells
    in $\widetilde X$ and $r_2^- \subset\boundary r_1$ 
then $r_1^+ \prec r_2^+$.
\end{enumerate}
\end{defn}

\begin{rem}\label{rem:2bislims}
The notion of `bislim' in \cite{HelferWise2015} imposes the additional requirement that
the 1-cell $r^+$ is \emph{uniquely strictly maximal} among 1-cells in $\partial r$.
More precisely, $r^+\not\preceq e$
for any other 1-cell $e$ in $\boundary r$, and $r^+$ is the unique such 1-cell in $\boundary r$.

Though Definition~\ref{defn:slim} appears more general, it is not hard to see that these two notions of bislim are equivalent. Indeed, if $X$ is bislim in the sense of Definition~\ref{defn:slim}, then one can alter the preorder to obtain unique strict maximality by restricting the preorder to $\{r^+:r\in \text{2-cells}\}$, and declaring $e\prec s^+$ for each 2-cell $s$ and $e\notin\{r^+\}$.

We caution that unique strict maximality is 
used for the notion of `slim'
and in the proof of \cite[Lem~5.2]{HelferWise2015}
and possibly in \cite[Lem~5.4]{HelferWise2015}.
\end{rem}

\begin{lem}\label{lem:unique strictly maximal}
Let $X$ be bislim. Let $\mathcal  E^+=\{r^+: r\in \text{2-cells}\}$. Then $r^+$ is uniquely strictly maximal among $\mathcal E^+\cap \partial r$, for each 2-cell $r$.
\end{lem}

When restricting attention to $\mathcal E^+$, one obtains an alternate $BS$ viewpoint.

\begin{rem}[$BS$ with $\prec$]\label{rem:pre/partial order}
By slightly altering the preorder $\preceq $, one can always assume that $\prec$ is a partial order on the set of 1-cells $\{r^+\}$. 
Indeed, replace $\preceq$ by the relation $\prec$ 
generated by:
 $\{r_1^+ \prec r_2^+  : 
\text{when
 $r_1\neq r_2$ and    $r_1^+ \subset \boundary r_2$ 
 or $r_2^-\subset \boundary r_1$}
 \}.$
Observe that $\prec$ is a partial order. 
This is a very general fact: given a relation $\leq$, a sub relation generated by 
elements of the form $\lneq$ yields a partial order.
(We are considering a subgraph $A\subset B$ of a digraph
where each edge of $A$ is not in any directed cycle. Hence $A$ has no directed cycle.)
\end{rem}

\begin{lem}\label{lem:bislim embedded 2-cells}
If $X$ is bislim, the 2-cells of $X$ embed in $\widetilde X$.
\end{lem}

\begin{proof}
If $\boundary_\p r$ is not a simple cycle in $\widetilde X$, then we choose a reduced disk diagram $D\rightarrow X$ so that $\partial_\p D$ is a nontrivial proper subpath of $\partial_\p r$.

If $D$ has no 2-cells, then $\partial_\p r$ is not immersed. 

Otherwise, we use $D$ to find a sequence of 2-cells $s_1,\ldots,s_n$ with
\[
s_n^+\prec s_1^+ \prec s_2^+ \prec \cdots \prec s_n^+~,
\]
as follows:
Choose a 2-cell $t$ of $D$. 
There is a 2-cell $s$ so that $\partial_\p s$ traverses $t^-$, whence $s^+\prec t^+$, and there is a 2-cell $u$ so that $\partial_\p u$ traverses $t^+$, whence $t^+\prec u^+$. 
Continuing forward and backward in this way, we find a path of 2-cells that either closes, or begins and ends at $r$. In either case, we obtain the contradiction above.
\end{proof}

\section{Good Stackings and Acyclic Graphs}

\begin{defn}[Dual graph]
The \emph{dual graph} $\Gamma$  of a 2-complex $X$ is defined as follows:
There is a vertex in $\Gamma$ for each 2-cell of $X$.
If $r_1$ and $r_2$ are 2-cells whose boundaries share a 1-cell $e$,
then $\Gamma$ has an associated edge whose endpoints are the corresponding vertices $v_1$ and $v_2$.
Note that when $n$ 2-cells share a 1-cell $e$,
all the edges in a complete graph on $n$ vertices are dual to $e$
(there are no dual edges when $n\leq 1$). See Figure~\ref{fig:DualGraph} for an example.
\end{defn}

\begin{figure}
\centering
\includegraphics[width=6cm]{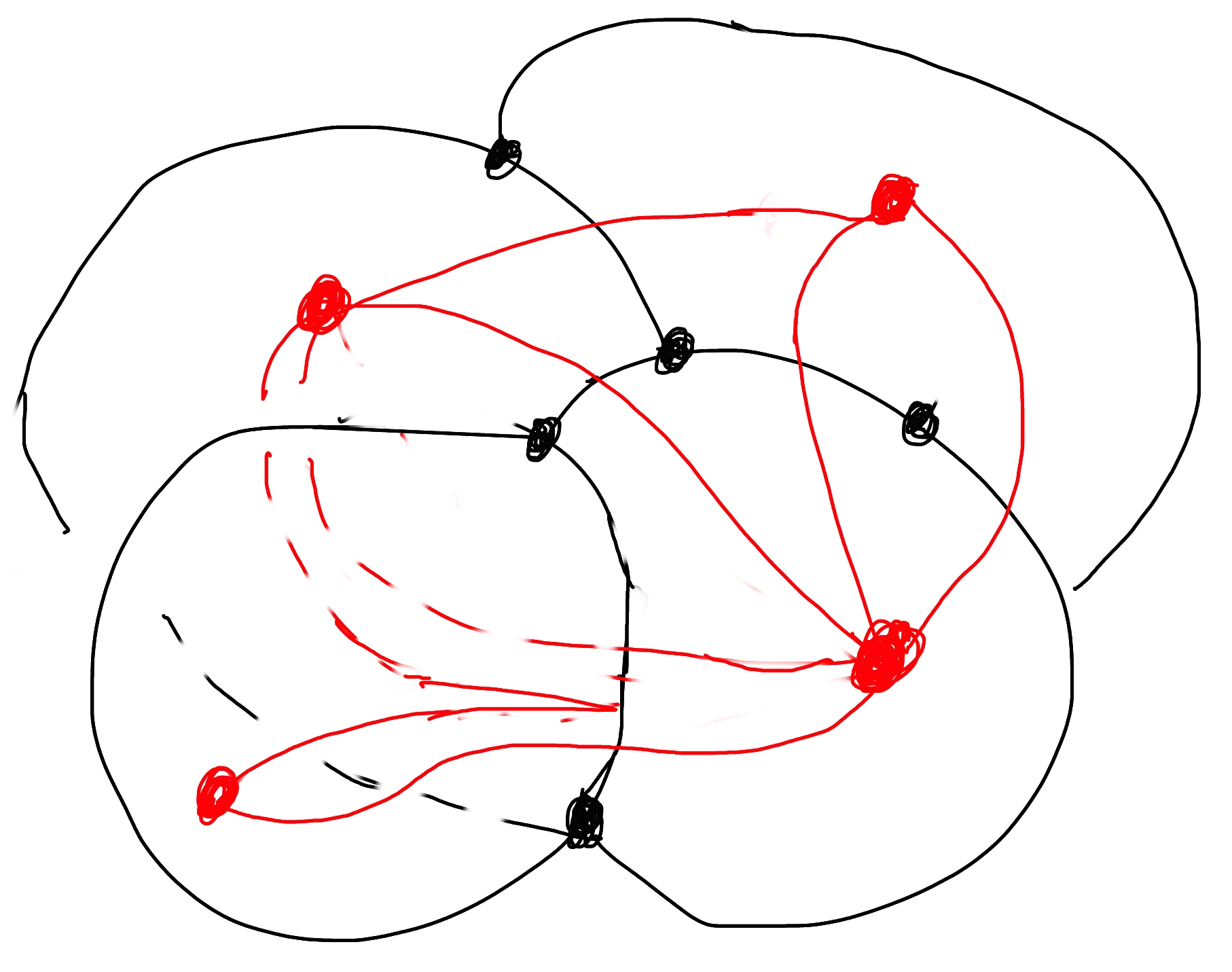}
\caption{The (red) dual graph of a (black) 2-complex.}
\label{fig:DualGraph}
\end{figure}

\begin{defn}[Directing the dual]
Suppose  $ X$ has a  stacking.
We direct its dual graph $\Gamma$ so that
if the side of the 2-cell $r_1$ at $e$ is lower than the side of the 2-cell $r_2$ at $e$ then
 the associated edge is directed  from $v_1$ to $v_2$.
\end{defn}

A digraph is \emph{acyclic} if it has no directed cycle.

\begin{lem}\label{lem:acyclic disks}
Let $D$ be a disk diagram with a good stacking.
Then $\Gamma$ is acyclic.
\end{lem}

\newcommand{\C}{\mathcal{C}}

\begin{proof}
Observe that there is an embedding $\Gamma\rightarrow D$ that sends vertices to 2-cells and edges to paths crossing 1-cells.
Suppose $\Gamma$ has a directed cycle $\C$.
Choose $\C$ so that the interior $R$ of the region inside $\C$ contains as few 
closed 1-cells of $D$ as possible.
Observe that each 1-cell of $D$ within $\C$ is dual to an edge of $\Gamma$. Moreover, each 2-cell of $D$ within $\C$ has no 1-cell on $\partial D$, and therefore has an incoming and an outgoing edge of $\Gamma$ at the corresponding dual vertex. 
Indeed, let $r$ be such a 2-cell, with dual vertex $v$.
Observe that over $\boundary_\p r$, the high/low points $h,\ell$ lie on distinct 1-cells $e_h,e_\ell$ of $D$.
Consequently, the edge dual to $e_h$ is directed towards the vertex $v$ dual to $r$, and the edge dual to $e_\ell$ is directed  from $v$.

Suppose that there is a closed 1-cell $e$ of $D$, within $\C$. As $e$ is internal, there is a directed edge of $\Gamma$ dual to $e$. Following this directed edge forwards and backwards, we find a directed path $P$ in $\Gamma$ that starts and ends on $\C$, through $e$. Combining $P$ with part of $\C$ yields a cycle $\C'$ in $\Gamma$ that contains fewer closed 1-cells than $\C$.

Finally, suppose that there are no closed 1-cells of $D$ within $\C$. In that case, $\C$ surrounds a single 0-cell $p$ of $D$.
However, the good stacking on $D$ induces a total order on corners of 2-cells at $p$ consistent with the directed edges between dual vertices in $\Gamma$ around $p$.
This total order contradicts the assumption that $\C$ is a directed loop,
unless $\C$ is a loop, and there is exactly one corner at $p$.
But this corresponds to the  situation of a backtrack in the boundary path of a 2-cell. 
This is excluded, since by hypothesis the boundary path of each 2-cell is immersed. 	   
\end{proof}

\begin{exmp} Lemma~\ref{lem:acyclic disks} can fail
for stackings that are ``half good'' in the sense that they only have a high. We refer to Figure~\ref{fig:BadOneSidedStacking}.
\end{exmp}
\begin{figure}
  \centering
  \includegraphics[width=.3\textwidth]{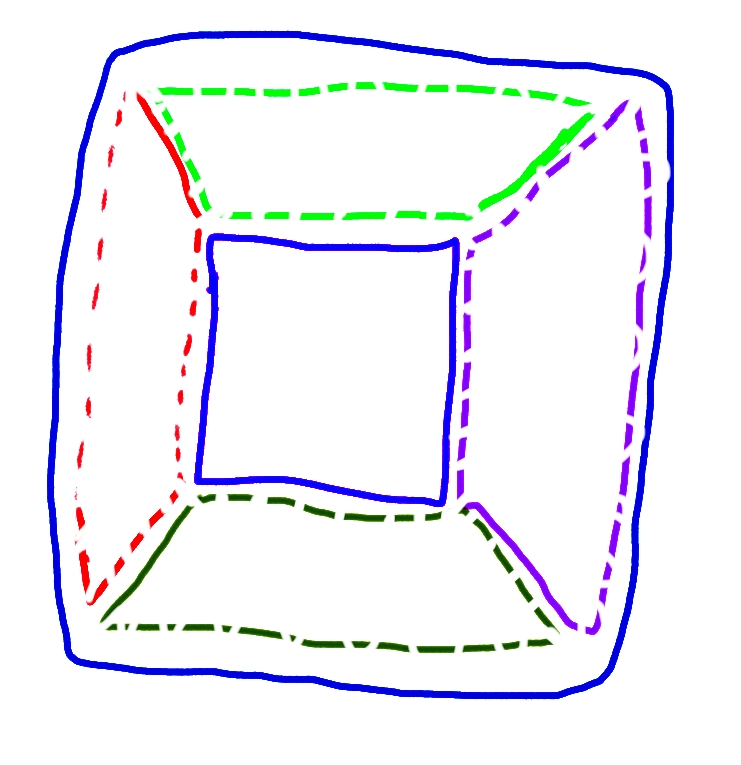}
  \caption{\label{fig:BadOneSidedStacking} In the stacking of a cube above
bold is higher than dotted. Each 2-cell cycle has a high, but might not have a low.
 Its dual graph has a directed cycle.}
\end{figure}

\begin{thm}\label{thm:acyclic universal}
Let $\widetilde X$ be a simply-connected 2-complex with a good stacking.
Its directed dual  $\Gamma$ is acyclic.
\end{thm}
\begin{proof}
Let $\C$ be a directed cycle in $\Gamma$.
Glue the consecutive 2-cells dual to vertices of $\C$ together along 1-cells dual to edges of $\C$ to form an annulus or M\"obius strip $A$. 
Choose a disk diagram $D\rightarrow \widetilde X$ so that $\boundary_\p D$ generates $\pi_1A$.
Moreover, we may assume that $D$ has a minimal number of 1-cells among all tuples $(\C,A,D)$.

Observe that $D$ is reduced, so edges of the dual of $D$ map to edges of $\Gamma$, and hence the dual of $D$ is directed.
Moreover, for each non-isolated 1-cell $e$ on $\partial_\p D$, the incident 2-cells in $D$ and $A$ do not form a cancellable pair along $e$.
Hence the 1-cell dual to $e$ maps to a 1-cell of $\widetilde X$, and is also directed. 

Suppose $e$ is an isolated 1-cell of $D$. 
If the 2-cells of $A$ that meet along $e$ in $A\cup_{\partial_\p D}D$ do not form a cancellable pair along $e$, we would obtain a shorter directed cycle that bounds a disk diagram with fewer 1-cells (a component of $D-e$).
Likewise, if the 2-cells that meet along $e$ do form a cancellable pair, we obtain a shorter directed cycle and disk diagram with fewer 1-cells.

Consider a non-isolated 1-cell of $D$. 
It is dual to an edge that extends to a directed arc $P\rightarrow A\cup_{\partial_\p D}D$ that starts and ends on vertices dual to 2-cells in $A$. 
Here we use that there are no cancellable pairs between $D$ and $A$ to ensure that $P$ starts and ends on $\C$. 
We obtain a directed cycle by combining $P$ with part of $\C$, and the result is a smaller disk diagram.

There is a base case: $D$ may contain zero 1-cells.
In this case $\C$ travels around a single vertex of $\Gamma$, contradicting the total order of 2-cells incident to vertices induced by a good stacking, as in the proof of Lemma~\ref{lem:acyclic disks}.
\end{proof}

\begin{cor}\label{cor:embeds}
Let $X$ be a 2-complex with a good stacking.
Then $\boundary_\p \tilde r \rightarrow \widetilde X$ is an embedding 
for each 2-cell $\tilde r$ of the universal cover $\widetilde X$ of $X$.
\end{cor}
\begin{proof}
As in Lemma~\ref{lem:bislim embedded 2-cells},
we choose a reduced disk diagram $D\rightarrow X$ so that $\partial_\p D$ is a nontrivial proper subpath of $\partial_\p r$. 

If $D$ has an interior 1-cell $e$, then we form a directed arc $P$, passing through $e$, that starts and ends on vertices dual to non-interior 2-cells $r_1$ and $r_2$ of $D$. 
We claim that $P$ extends to a directed cycle in the directed dual graph of $\widetilde X$, contradicting Theorem~\ref{thm:acyclic universal}.
Indeed, if both $r_1$ and $r_2$ do not form a cancellable pair with $r$, this is evident.
If $r_1$ does but $r_2$ does not form a cancellable pair with $r$, then $P$ actually starts on the vertex dual to $r$, but extends to a directed path that ends on $r$ as well.
When both $r_1$ and $r_2$ form a cancellable pair with $r$, $P$ is a directed path that starts and ends on $r$. 
Note that $P$ is nontrivial, as it traverses the edge dual to $e$, which is between a non-cancellable pair of 2-cells.

Finally, suppose $D$ has no interior 1-cell.
Then $D$ is a 2-cell which cannot cancel with $r$, since $\partial_\p D$ is a proper subpath of $\partial_\p r$.
We obtain a directed cycle of length two in the dual graph, as in the case above where neither $r_1$ or $r_2$ form a cancellable pair with $r$, again contradicting Theorem~\ref{thm:acyclic universal}.
\end{proof}

\begin{rem}
An acyclic digraph has a partial order on its vertices where
$a\prec b$  if and only if there is a directed path from $a$ to $b$.
Thus
Theorem~\ref{thm:acyclic universal} provides a partial order on the 2-cells of $\widetilde X$. 
\end{rem}

\begin{lem}[$GS\Rightarrow BS$]
\label{lem:GS is BS}
If $X$ has a good-stacking, then $X$ is bislim.
\end{lem}

\begin{proof}
 Theorem~\ref{thm:acyclic universal} provides a $\pi_1X$-invariant partial order on the 2-cells of $\widetilde X$. This induces a $\pi_1X$-invariant partial order on the 1-cells as follows:
For each 1-cell $e$ of $\widetilde X$,  declare its position in the partial order to be
the same as the 2-cell whose side over $e$ is lowest.

For a 2-cell $r$, declare $r^+$ to be a 1-cell on $\boundary r$ where the side of $r$ is lowest among all sides.
Declare $r^-$ to be a 1-cell over which $r$ has a highest side.
Moreover, we do this $\pi_1X$-equivariantly.

Corollary~\ref{cor:embeds} asserts
that  $\boundary_\p r$ embeds in  $\widetilde X$,
and so
$r$ traverses $r^+$ exactly once.

If $r_1^+=r_2^+$ then $r_1=r_2$, since at most one 2-cell has a lowest side over a 1-cell.

If $r_2^- \subset \boundary r_1$
  then
there is an edge directed from the vertex dual to $r_1$ to the vertex dual to $r_2$,
and hence $r_1^+\prec r_2^+$.
Similarly, if $r_1^+ \subset \boundary r_2$ then $r_1^+\prec r_2^+$
since there is an edge directed from the vertex dual to $r_1$ to the vertex dual to $r_2$.
\end{proof}

\begin{exmp}\label{exmp:slim not GS}
Note that there are slim complexes without a good stacking.
The easiest example is $\langle a,b \mid a, ab\rangle$.
We order the 1-cells by $a\prec b$.
We declare $r_1^+=a$.
We declare $r_2^+=b$.
\end{exmp}

\section{Totally-ordered Invariant Staggering}

The goal of this section is to introduce a generalization of staggered 2-complexes. 
Later, in Section~\ref{sec:pes}, we will offer a weaker but more involved version.

A \emph{staggered 2-complex} is a 2-complex $X$ whose 2-cells have immersed attaching maps, such that there is a total ordering
on the set of 2-cells, and a total ordering on a subset $\mathcal E$ of the 1-cells with the following properties:
Each 2-cell $\alpha$ has an $\mathcal E$ 1-cell in $\boundary \alpha$.
If $\max(\alpha)$ and $\min(\alpha)$ are the maximal and minimal $\mathcal E$ 
1-cells in $\boundary\alpha$, then
$\alpha\prec \beta$ implies $\max(\alpha)\prec \max(\beta)$ and $\min(\alpha)\prec \min(\beta)$.

Every 2-complex associated to the presentation of a one-relator group $\langle a,b \mid w\rangle$ is staggered, and the notion of staggered arises naturally in the theory of one-relator groups \cite{LS77, Howie87}.

\begin{rem}\label{rem:staggered implies GS,BS}
If $X$ is a staggered 2-complex without torsion, then $X$ admits a good stacking \cite{LouderWilton2014}, and $X$ is bislim \cite[Prop~2.4]{HelferWise2015}.
(The proof of bislimness there 
is correct using Definition~\ref{defn:slim}, but must be altered to ensure unique strict maximality, as in Remark~\ref{rem:2bislims}.)
\end{rem}

We begin by offering the following preliminary generalization of staggered:

\begin{defn}[$TIS$]\label{defn:TIS}
A \emph{totally ordered invariant staggering} of a 2-complex $X$, whose 2-cells have immersed attaching maps, consists of the following:
\begin{enumerate}
\item\label{item:TES order 2-cell} 
A $\pi_1X$-invariant ordering $\prec$ on  2-cells of $\widetilde X$.
\item\label{item:TES order} 
A $\pi_1X$-invariant ordering $\prec^+$ on a subset $\mathcal E^+$ of the 1-cells of $\widetilde X$.
\item A $\pi_1X$-invariant ordering $\prec^-$ on a subset $\mathcal E^-$ of the 1-cells of $\widetilde X$.
\item\label{cond:max/min} For each 2-cell $\alpha$, there are $\mathcal E^\pm$ 1-cells in $\partial \alpha$. \\
Let $\max^+(\alpha)$ denote the $\prec^+$-highest 1-cell in $\boundary \alpha$.\\
Let $\min^-(\alpha)$ denote the $\prec^-$-lowest  1-cell in $\boundary \alpha$. 
\item\label{cond:max traversed once} The 1-cell $\max^+(\alpha)$ is traversed exactly once by $\boundary_\p \alpha$.
\item\label{cond:consistent} If  $\alpha\prec \beta$ then
$\max^+(\alpha)\prec^+\max^+(\beta)$ and
$\min^-(\alpha)\prec^- \min^-(\beta)$.
\end{enumerate}
\end{defn}

\begin{exmp}
\label{exmp:staggered}
A staggered 
2-complex $X$ without torsion yields a totally ordered invariant staggering as follows: 
We begin by choosing a left-order on $\pi_1X$, whose existence holds by Burns-Hale since $\pi_1X$ is locally indicable \cite{BurnsHale72,Brodskii84,Howie00}.
We now declare the ordering $\prec$ on 2-cells of $\widetilde X$.
If $\widetilde\alpha$ and $\widetilde\beta$ are 2-cells in distinct $\pi_1X$-orbits, declare $\widetilde\alpha\prec \widetilde\beta$ when their images satisfy $\alpha \prec \beta$ in the staggering of $X$.
If $\widetilde \beta = g\cdot \widetilde \alpha$, for $g\in\pi_1X$, declare $\widetilde \alpha \prec \widetilde \beta$ if $g>1$ in the chosen left-order.
The relation $\prec$ described here induces a total-order on the 2-cells. 

Now let $\mathcal E$ be the ordered 1-cells of $X$, and let $\mathcal E^+=\mathcal E^-$ be their preimages. 
Let $\widetilde a$, $\widetilde b$ be in $\mathcal E^\pm$. 
Then $\widetilde a \prec^+ \widetilde b$ and $\widetilde a \prec^- \widetilde b$ provided their images satisfy $a\prec b$.
(Condition~\eqref{cond:max traversed once} holds as 2-cells of the staggered $X$ embed in $\widetilde X$ by the \emph{Freiheitsatz} for staggered 2-complexes, together with the fact that proper subwords of a relator are nontrivial in one-relator groups \cite{Weinbaum71}.)

Of course, by left-orderability, $TIS$ is not possible in the presence of torsion. 
\end{exmp}

\begin{rem}
The converse implication of Example~\ref{exmp:staggered} does not hold. Indeed, Figure~\ref{fig:not staggered} provides an example whose quotient 2-complex $X$ has fundamental group $\langle a,b,c \ | \ abc^{-1},bac^{-1}\rangle$. Because both relator words contain each generator, the complex $X$ is not staggered. On the other hand, $X$ has $TIS$. Indeed, as pictured, one can provide a total order on all cells by projecting the barycenters orthogonally to a line of irrational slope, 
and the conditions of $TIS$ can be verified as an exercise. 

In the example pictured in Figure~\ref{fig:not staggered}, one could choose two distinct irrationally-sloped lines to obtain two distinct orders $\prec^-$ and $\prec^+$.

If instead one began with a line of rational slope, then the comparisons obtained by projection would only yield a partial order. 
For most rational slopes, one obtains an \emph{invariant staggering}, which we define carefully and investigate in Section~\ref{sec:pes}.
\end{rem}

\begin{figure}
\centering
\includegraphics[width=6cm]{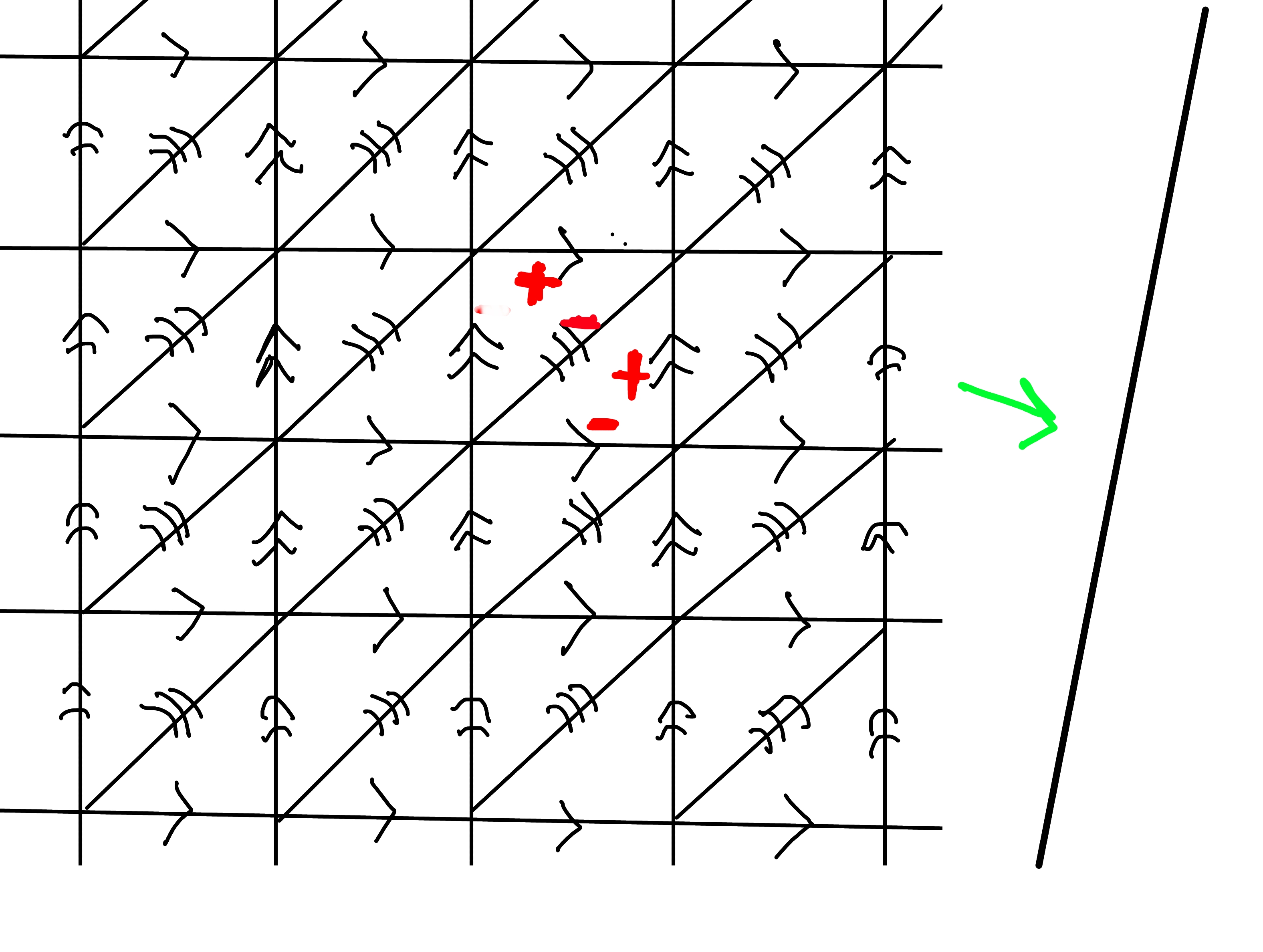}
\caption{A 2-complex with $TIS$ but whose quotient is not staggered.}
\label{fig:not staggered}
\end{figure}

\begin{prob}
The classical case uses a single ordering on a subset of the 1-cells of $X$.
Does the classical case generalize to two distinct orderings on the 1-cells of $X$?
Is there an example that is staggered using two orderings but not one? 
\end{prob}

\begin{rem}
In the motivating case, $\prec^+$ and $\prec^-$ are identical.
Warning: In general, we do not know that
$\min^+(\alpha) \prec^+ \min^+(\beta)$ and 
$\max^-(\alpha) \prec^- \max^-(\beta)$.
(Otherwise, we would have used just one ordering...!)
\end{rem}

\begin{lem}
\label{lem:TES 2-cells embed}
If $X$ is $TIS$, then for each 2-cell $r\rightarrow \widetilde X$, the path $\partial_\p r \rightarrow X$ is an embedding.
\end{lem}

\begin{proof}
A $2$-complex $X$ is \emph{2-collapsing} \cite[Def.~5.7]{GasterWise2019}
if the following holds:
For each compact subcomplex $Y \subset \widetilde X$, 
if $Y$ has at least $m\leq 2$ 2-cells, then $Y$ has at least $m$ distinct collapses along free faces.
We will show that $TIS$ implies that $X$ is 2-collapsing, and the desired conclusion follows from \cite[Lem.~5.8, Cor.~3.7]{GasterWise2019}.

Let $Y\subset X$ be a compact subcomplex. 
If $Y$ has at least two 2-cells, then we may choose $\alpha$ and $\beta$ to be the $\prec$-maximal and $\prec$-minimal 2-cells, respectively.
Now $\alpha$ collapses along $\max^+(\alpha)$ by Condition~\eqref{cond:consistent}, and similarly $\beta$ collapses along $\min^-(\beta)$, so $Y$ has two collapses.
If $Y$ has a single 2-cell $\alpha$, then it collapses along $\max^+(\alpha)$. 
\end{proof}

\begin{defn}[$TBS$]
\label{def:TBS}
A bislim structure on $X$ is \emph{totally ordered}
if the partial order on the set $\{\alpha^+\}$  is a total ordering.
\end{defn}

\begin{lem}[$TBS \Leftrightarrow TIS$]\label{lem:TBS gives TIS}
$X$ has a totally ordered bislim structure if and only if $X$ has a totally ordered invariant staggering.
\end{lem}
\begin{proof}[Proof of $TBS\Rightarrow TIS$]
Let $\mathcal E^+ =\{\alpha^+ : \alpha \in \text{$2$-cells} \}$.
Let $\mathcal E^- =\{\alpha^- : \alpha \in \text{$2$-cells} \}$.

Let  $<$ denote the order on the distinguished $\alpha^+$ 1-cells
in the bislim structure.
Declare $\alpha\prec \beta$ and
 $\alpha^+\prec^+ \beta^+$ and
  $\alpha^-\prec^- \beta^-$ precisely when $\alpha^+ <  \beta^+$.

Observe that $\max^+(\alpha)=\alpha^+$ and hence
$\max^+(\alpha) \prec^+ \max^+(\beta)$ when $\alpha \prec \beta$.

Observe that $\min^-(\alpha)=\alpha^-$ since
if $\beta^-$ is traversed by $\boundary \alpha$ then $\alpha^+ < \beta^+$, and hence
$\alpha^- \prec^- \beta^- $.
\end{proof}
\begin{proof}[Proof of $TIS\Rightarrow TBS$]
Suppose that $X$ has a totally ordered invariant staggering.

Declare $r^+=\max^+(r)$ and $r^-=\min^-(r)$ for each 2-cell $r$.

Condition~\ref{defn:slim}.\eqref{slim:trav once} is implied by Lemma~\ref{lem:TES 2-cells embed}.

If $r^+\subset \boundary s$ and $r\neq s$
then Condition~\ref{defn:TIS}.\eqref{cond:max/min} implies $r^+\prec^+ s^+$. 

If $s^-\subset \boundary r$ and $r\neq s$
then $s^-\prec^- r^-$ and so $s\prec r $ and so $s^+\prec^+ r^+$.

This demonstrates that $TIS\Rightarrow BS$, but in fact the bislim structure contructed above is totally ordered by Condition~\ref{defn:TIS}.\eqref{item:TES order}.
\end{proof}

\begin{prob}
\label{prob:BS implies TIS}
Is $TIS$ equivalent to $BS$? (Apologetically, for 2-complexes with countably many 2-cells.)
\end{prob}

A group $G$ is \emph{Szpilrajneous} if a $G$-invariant  partial ordering on a set
extends to a $G$-invariant total ordering on that set, provided $G$ acts freely.
The trivial group is Szpilrajneous by \cite{Szpilrajn1930}, as a consequence of Zorn's Lemma.
Any free-abelian group is Szpilrajneous by Pruss \cite[Thm~2]{Pruss2014}.
The Klein bottle fundamental group is not Szpilrajneous \cite[discussion after Thm.~9]{DDHPV2007}. 

\begin{prob}
Which groups are Szpilrajneous?
\end{prob}

For Szpilrajneous groups, Problem~\ref{prob:BS implies TIS} has a positive answer.

\begin{cor}  Suppose $G$ is Szpilrajneous.
Then for the presentation 2-complex of $G$ we have $BS \Rightarrow TBS$.
Hence $BS\Leftrightarrow TIS$ when $G$ is Szpilrajneous.
\end{cor}

\begin{proof}
In a bislim structure on $X$, the set $\{r^+ :  r \in \text{2-cells}\}$ is partially ordered.
This extends to a total ordering on the full set of 1-cells by Szpilrajneousness.
Finally, $TBS\Rightarrow BS$, and we conclude with Lemma~\ref{lem:TBS gives TIS}.
\end{proof}

\section{Invariant Staggering}
\label{sec:pes}
We conclude by introducing invariant staggerings, a relaxation of Definition~\ref{defn:TIS}.

\newcommand{\Max}{\mathrm{Max}}
\newcommand{\Min}{\mathrm{Min}}	

\begin{defn}[$IS$]
\label{def:ES}
	An \emph{invariant staggering} on a 2-complex $X$ whose 2-cells have immersed attaching maps consists of the following:
	\begin{enumerate}
\item A $\pi_1 X$-invariant partial order $\prec$ on 2-cells of $\widetilde{X}$;

\item A $\pi_1 X$-invariant partial order $\prec^+$ on a subset $\mathcal{E}^+$ of the 1-cells of $\widetilde{X}$;
		
\item\label{cond:<-} A $\pi_1 X$-invariant partial order $\prec^-$ 
    on a subset $\mathcal{E}^-$ of the 1-cells of $\widetilde{X}$
    
\item \label{cond:pes-sufficiency} 
$\boundary \alpha$ contains 1-cells of $\mathcal E^{\pm}$ for each 2-cell $\alpha$.

\noindent
Let $\Max^+(\alpha)$ denote the set of elements in $\partial \alpha \cap \mathcal{E}^+$
that are $\prec^+$-maximal.\\
\noindent
Let $\Min^-(\alpha)$ denote the set of elements in $\partial \alpha\cap \mathcal{E}^-$
that are $\prec^-$-minimal.
		
\item \label{cond:pes-unique-traversal} 
Each 1-cell of  $\Max^+(\alpha)$ is traversed exactly once by $\boundary_\p \alpha$, for each $\alpha$.

	\item \label{cond:pes-two-cell-order}If $\alpha \prec \beta$ then $\Max^+(\alpha) \prec^+ \Max^+(\beta)$ and $\Min^-(\alpha) \prec^- \Min^-(\beta)$, where $S < T$ denotes that $s < t$ for each $s \in S$ and $t \in T$.
		
\item \label{cond:pes-max} If $\alpha\neq \beta$
and a 1-cell of $\Max^+(\alpha)$ lies in $\partial \beta$, then $\alpha \prec \beta$.
		
\item \label{cond:pes-min} 
If $\alpha\neq \beta$ and a $1$-cell of $\Min^-(\alpha)$ lies in $\partial \beta$, then $\beta \prec \alpha$.

\end{enumerate}
\end{defn}

\begin{rem}
In the definition above, we could also have replaced $\Max^+(\alpha)$ and $\Min^-(\alpha)$ with chosen extremal 1-cells (chosen $\pi_1X$-equivariantly).
The proofs below would be nearly unchanged, and the conditions would look slightly simpler though somewhat less natural.
\end{rem}

\begin{rem}
Let $ X$ be a 2-complex with $TIS$, and let $\dot X$ be obtained from $X$ by raising 2-cells to proper powers. The totally ordered invariant staggering of $X$ almost induces an invariant staggering of $\dot X$, with the caveat that Conditions~\ref{def:ES}.\eqref{cond:pes-max}~and~\ref{def:ES}.\eqref{cond:pes-min} become slightly problematic, cf.~Remark~\ref{rem:torsion}.
Indeed, consider the map $\widetilde {\dot X} \rightarrow \widetilde X$, and order cells in the domain by their images in the target. 
In particular, staggered 2-complexes with torsion are (almost) $IS$. 
Note that each 2-cell now has multiple maximum 1-cells which are mutually incomparable.
\end{rem}

We observe that $IS$ is weaker than $TIS$:

\begin{lem}[$TIS\Rightarrow IS$]
\label{lem:TES to IS}
If $X$ has a totally ordered invariant staggering, then $X$ has an invariant staggering.
\end{lem}

\begin{proof}
Conditions~\eqref{cond:pes-max}~and~\eqref{cond:pes-min} are implied by an invariant staggering. 
Indeed, suppose $\Max^+(\alpha) \subset \partial \beta$. 
Then $\Max^+(\alpha) \prec^+ \Max^+(\beta)$ by definition, and, as $\prec$ is a total order on 2-cells, we have $\alpha \prec \beta$.
\end{proof}

It is not hard to see that $IS$ is equivalent to $BS$.

\begin{lem}[$BS\Leftrightarrow IS$] 
\label{lem:BS implies IS}
$X$ has an invariant staggering if and only if $X$ is bislim. \end{lem} 

\begin{proof}[Proof of $IS \Rightarrow BS$]
For each 2-cell $r$, we declare $r^+$ to be a 1-cell of $\Max^+(r)$ and $r^-$ to be a 1-cell of $\Min^-(r)$. Moreover, we insist that these choices are made $\pi_1$-equivariantly.

If $r^+\subset \boundary s$ and $r\neq s$
then Condition~\eqref{cond:pes-max} implies $r\prec s$, and hence Condition~\eqref{cond:pes-two-cell-order} implies $r^+\prec^+ s^+$.

If $s^-\subset \boundary r$ and $r\neq s$
then Condition~\eqref{cond:pes-min} implies $s\prec r $, and hence Condition~\eqref{cond:pes-two-cell-order} implies $s^+\prec^+ r^+$.
\end{proof}

\begin{proof}[Proof of $BS \Rightarrow IS$]
By Remark~\ref{rem:pre/partial order}, 
we can assume the preorder of the $BS$ structure is a partial order, which we denote by $\prec^+$ for the purposes of this proof. 

Let $\mathcal E^\pm = \{ r^\pm: r\in \text{2-cells$(\widetilde X)$}\}$.
We restrict $\prec^+$ to $\mathcal E^+$.
By Lemma~\ref{lem:unique strictly maximal}, $\Max^+(r)=\{r^+\}$.

	Declare $a^- \prec^- b^-$ if and only if $a^+ \prec^+ b^+$.
Again, $\Min^-(r)=\{r^-\}$.
	
Declare $a \prec b$ if for $e_a \in \Max^+(a)$, $e_b \in \Max^+(b)$ we have $e_a \prec^+ e_b$.
	
	Conditon~\ref{def:ES}.\eqref{cond:pes-sufficiency} holds since
	 every 2-cell $r$ has distinguished edges $r^-$ and $r^+$.
	
	Conditon~\ref{def:ES}.\eqref{cond:pes-unique-traversal} is assumed in the definition of a bislim structure.
	 
	If the element $a^+$ of $\Max^+(a)$ lies in $\partial b$, then $a^+ \prec^+ b^+$ due to the bislim structure, whence $a \prec b$.
If the element $a^-$ of $\Min^-(a)$ lies in $\partial b$, then $b^+ \prec^+ a^+$ whence $b \prec a$.
	
	Finally, $a \prec b$ implies $a^+ \prec^+ b^+$ and $a^- \prec^- b^-$ by definition.
	\end{proof}

We now complete the proof of the equivalence of $IS$, $GS$, and $BS$.

\begin{lem}[$IS\Rightarrow GS$]
\label{lem:ES to GS}
If $X$ has an invariant staggering then $X$ has a good stacking.
\end{lem}

\begin{proof}
We first show that there is a good stacking of $\widetilde X$ that is $\pi_1X$-equivariant.
This descends to a stacking of $X$.
(We use that $X$ is locally countable. Otherwise, a larger totally ordered set must be substituted for $\reals$ in the definition of good stacking.)

Let $\widetilde W$ be the disjoint union of the boundary cycles of the 2-cells of $\widetilde X$.
Let $\widetilde \sigma : \widetilde W\rightarrow \peals$, where $\peals $ is the partially ordered set of 2-cells. This is the quotient map sending each component to a point.

Let $\widetilde \Phi:\widetilde W \rightarrow \widetilde X^1$ be the disjoint union of attaching maps.
This determines a $\pi_1X$-equivariant map $\widetilde \Psi: \widetilde W \rightarrow 
\widetilde X^1 \times \peals$, such that $\rho_{\widetilde X^1} 
\circ \widetilde \Psi = \widetilde \Phi$
and $\rho_\peals \circ \widetilde \Psi = \widetilde \sigma$.
For each 2-cell $r$, observe that each 1-cell $c$ of $\Max^+(r)$ provides `low points' for $\boundary r$, in the sense that, among all 2-cells $s\neq r$ whose boundary paths traverse $c$, we have $\widetilde \sigma(\partial r) \prec \widetilde \sigma(\partial s)$: indeed, if $c\subset \partial s$ then $c \prec^+ \Max^+(s)$, 
by Condition~\ref{def:ES}.\eqref{cond:pes-max}.
Similarly, $\Min^-(r)$ provides `high points' for $\boundary r$ by Condition~\ref{def:ES}.\eqref{cond:pes-min}.

As $\peals$ is countable, Szpilrajn's Theorem \cite{Szpilrajn1930} provides an order-preserving inclusion $\peals \to \reals$.
Endow $\peals$ with the subspace topology.

By Lemma~\ref{lem:BS implies IS}, $X$ is bislim, and therefore by Lemma~\ref{lem:bislim embedded 2-cells} its 2-cells embed in $\widetilde X$. 
Therefore $\widetilde \sigma : \widetilde W \rightarrow \widetilde X^1\times \peals$ is an embedding, which yields an embedding $\sigma:W \rightarrow 
\pi_1X\backslash (\widetilde{X}^1\times \peals)$.
The embedding $\peals\to \reals$ yields the commutative diagram:
\[\xymatrixcolsep{.3pc}
\xymatrix{
& \pi_1X\backslash (\widetilde{X}^1\times \peals)\ar[rr] \ar[rd]
& & \pi_1X\backslash (\widetilde X^1 \times \reals) \ar[ld] \\
&& X^1 
}
\]
The right side is an orientable $\reals$-bundle over $X^1$ because $\pi_1X \curvearrowright\peals$ is $\prec$-preserving. Therefore 
$\pi_1X\backslash (\widetilde{X}^1\times \reals)\approx X^1\times \reals$, and $\sigma$ yields a stacking
which is good by Condition~\ref{def:ES}.\eqref{cond:pes-sufficiency}: for each 2-cell $r$, there are high and low points for $\sigma(\partial r)$.
\end{proof}

\begin{rem}\label{rem:torsion}
A more powersful version of Theorem~\ref{main thm} holds, in which relators of $X$ are allowed to be proper powers.
Good stackings should be replaced by pretty good stackings, that allow relators to cover their images.
The definition of bislim should be altered as follows: replace `distinct 2-cells' with `$\partial r_1 \neq \partial r_2$' in Conditions~\ref{defn:slim}.\eqref{slim:distinct}~and~\ref{defn:slim}.\eqref{bislim:bislim}.
The definition of invariant staggering should be altered as follows: replace `$\alpha\neq \beta$' with `$\partial \alpha \neq \partial \beta$' in Conditions~\ref{def:ES}.\eqref{cond:pes-max}~and~\ref{def:ES}.\eqref{cond:pes-min}. 
\end{rem}
	
\vspace{1cm}

\bibliographystyle{alpha}
\newcommand{\etalchar}[1]{$^{#1}$}
\def\cprime{$'$} \def\polhk#1{\setbox0=\hbox{#1}{\ooalign{\hidewidth
  \lower1.5ex\hbox{`}\hidewidth\crcr\unhbox0}}} \def\cprime{$'$}
  \def\cprime{$'$} \def\polhk#1{\setbox0=\hbox{#1}{\ooalign{\hidewidth
  \lower1.5ex\hbox{`}\hidewidth\crcr\unhbox0}}}

\newpage

\end{document}